\documentclass[11pt,a4paper,reqno]{amsart}
\usepackage{amsmath, amssymb, latexsym, enumerate,  graphicx,tikz, stmaryrd, eufrak,enumitem}

\usepackage{thmtools}

\usepackage[pdftex]{hyperref}
\usepackage{cleveref}
\usepackage{bbm}
\usepackage{cases}
\usepackage{array}
\usepackage{tikz-cd}
\usepackage[all]{xy}

\usepackage[hmarginratio={1:1},vmarginratio={1:1},lmargin=80.0pt,tmargin=90.0pt]{geometry}

\usepackage{tikz}
\tikzset{anchorbase/.style={baseline={([yshift=-0.5ex]current bounding box.center)}}}
\usetikzlibrary{decorations.markings}
\usetikzlibrary{decorations.pathreplacing}
\usetikzlibrary{arrows,shapes,positioning,backgrounds}
\tikzstyle directed=[postaction={decorate,decoration={markings,
    mark=at position #1 with {\arrow{>}}}}]
\tikzstyle rdirected=[postaction={decorate,decoration={markings,
    mark=at position #1 with {\arrow{<}}}}]
    
\usepackage{graphicx}
\usepackage{nicefrac}

 \newlength{\baseunit}               
 \newcount{\numlines}                
\newtheorem{theorem}[subsubsection]{Theorem}
\newtheorem{lemma}[theorem]{Lemma}
\newtheorem{prop}[theorem]{Proposition}
\newtheorem{corollary}[subsubsection]{Corollary}

\theoremstyle{definition}

\newtheorem{remark}[theorem]{Remark}

\newtheorem{example}[subsubsection]{Example}

\numberwithin{equation}{section}

\newcommand{\E}{\mathrm{E}}

\newcommand{\GL}{\mathrm{GL}}

\newcommand{\mF}{\mathbb{F}}

\newcommand{\Spec}{\mathrm{Spec}}

\newcommand{\Ab}{\mathsf{Ab}}

\newcommand{\Rep}{\mathsf{Rep}}

\newcommand{\mN}{\mathbb{N}}
\newcommand{\mZ}{\mathbb{Z}}

\newcommand{\mC}{\mathbb{C}}

\newcommand{\End}{\mathrm{End}}

\newcommand{\perf}{\mathrm{perf}}

\begin{document}
\title[]{K-theory and perfection}
\author{Kevin Coulembier}

\date{\today}


\begin{abstract}
We review some recent results on $K$-theory of perfection of commutative $\mF_p$-algebras and provide an alternative proof.
\end{abstract}

\maketitle

\section{Introduction}
Let $p$ be a prime.
In \cite{KM}, Kelly and Morrow observe that  $K(-)[1/p]$ yields an equivalence when applied to a commutative $\mF_p$-algebra $R$ and its perfection $R_{\perf}$, by making use of nilpotence-invariance. In \cite{EK}, Elmanto and Khan prove the same result as an application of their topological invariance (after inverting $p$) in moticic homotopy theory.

In combination with the classical fact that the groups $K_n(R)$, for $n>0$, are uniquely $p$-divisible for a perfect $\mF_p$-algebra, see \cite[Corollaire~5.5]{Kratzer}, this implies
\begin{equation}\label{eqMain}K_n(R_{\perf}) \;\cong\; \mZ[1/p]\otimes K_n(R),\quad n>0.\end{equation}
This is reminiscent of the isomorphism
$$\pi_n(X_{\frac{1}{p}})\;\cong\; \mZ[1/p]\otimes \pi_n(X),\quad n>0,$$
for $p'$-localisations of topological spaces as in \cite{Sullivan}. Note that the analogy has its limitations, see Remark~\ref{RemDet}.

In \cite{CW}, another analogy between the two manipulations was studied. Namely, it was proved that the $p'$-localisations of classifying spaces of compact connected Lie groups behave similarly to the perfections of the corresponding reductive groups in characteristic $p$.

In fact, another connection between algebraic perfection and topological localisation can be concluded almost immediately from the above results on $K$-theory. Quillen's functor from rings to spaces
$$\mathbf{Q}:R\;\mapsto\; \mathrm{GL}(R)\;\mapsto\; B \mathrm{GL}(R)\;\mapsto \; (B\mathrm{GL}(R))^+,$$
sends perfections of $\mF_p$-algebras to $p$'-localisations, see Theorem~\ref{Thm+}. Explicitly, $\mathbf{Q}(R_{\perf})\simeq \mathbf{Q}(R)_{\frac{1}{p}}$.

In this short note we give a direct proof of the isomorphism~\eqref{eqMain}, by showing it can be derived from an exercise with Adams operations, see Theorem~\ref{Main}. We also explain what happens in negative $K$-theory.

\subsection{Notation}We set $\mN=\{0,1,2,\cdots\}$.
For an endomorphism $f\in\End(X)$ in some category, we write $\varinjlim_fX$ for the colimit (assuming it exists) of the diagram
$$X\xrightarrow{f}X\xrightarrow{f}X\xrightarrow{f}\cdots.$$

\section{Localisation and Adams operations}

\subsection{Localisation}

We fix an integer $\ell\in\mZ$.

\subsubsection{} Denote by $\Ab_{\ell}$ the full subcategory of $\Ab$ of abelian groups on which $\ell$ acts invertibly, in other words, $\Ab_{\ell}$ is the category of $\mZ[1/\ell]$-modules. Then we have the localisation functor
$$-[1/\ell]:\,\Ab\to\Ab_{\ell},\quad A\mapsto A[1/\ell]\cong \mZ[1/\ell]\otimes A\cong \varinjlim_{\ell} A,$$
which is the left adjoint of the inclusion functor.

\begin{lemma}\label{LemEll}
For an endomorphism $\theta: A\to A$ of an abelian group $A$, the following two conditions are equivalent:
\begin{enumerate}
\item 
\begin{enumerate}
\item The endomorphism $\theta[1/\ell]$ of $A[1/\ell]$ is an automorphism; and
\item for each $a\in A$ there exists $i\in\mN$ with $\theta^i(a)\in \ell A$; and
\item if $\ell a=0$ for some $a\in A$, then there exists $i\in\mN$ with $\theta^i(a)=0$.
\end{enumerate}
\item We have an isomorphism under $A$
$$A[1/\ell]\;\cong\; \varinjlim_{\theta}A.$$
\end{enumerate}
\end{lemma}
\begin{proof}
We prove that (1) implies (2).
From conditions (b) and (c), it follows that $\varinjlim_\theta A$ is in $\Ab_{\ell}$. Consequently
$$\varinjlim_\theta A\;\cong\; \varinjlim_{\ell}\varinjlim_\theta A\;\cong\;\varinjlim_{\theta[1/\ell]}(A[1/\ell])\;\cong\; A[1/\ell],$$
where the last isomorphism follows from assumption (a).

The other direction of the lemma, which we will not need, is proved similarly.
\end{proof}

\begin{example}
The conditions in Lemma~\ref{LemEll} are satisfied by the endomorphism 
$$\left(\begin{array}{cc}
\ell & 0 \\
1 & \ell^2
\end{array}\right):\mZ^2\to\mZ^2.
$$
\end{example}

\subsection{Adams operations}

\subsubsection{Set-up}\label{SetUp}

Let $\Lambda$ be a $\lambda$-ring, in the terminology of \cite{Yau}, that is a special $\lambda$-ring in the terminology of \cite{AT, Weibel}. Concretely, the ring $\Lambda$ is equipped with maps $$\{\lambda^n:\Lambda\to\Lambda\,|\, n\in\mN\},$$
satisfying properties (1)-(6) in~\cite[Definition~1.10]{Yau}. 
Consequently, for every $n\in\mZ_{>0}$, we have the $n$-th Adams operation 
$$\psi^n:\Lambda\to\Lambda,$$
which is a ring homomorphism.

We consider a binomial $\lambda$-subring $H_0\subset\Lambda$, for which we have a section $\varepsilon:\Lambda\to H_0$ of $\lambda$-rings (for example, $H_0$ is part of a positive structure, see \cite[Definition~II.4.2.1]{Weibel}).

We need to impose one condition on the above structure, which is that the associated {\em $\gamma$-filtration on $\Lambda$ is finite}. Explicitly, there exists $N\in\mN$ such that
$$\prod_i\lambda^{n_i}(x_i+n_i-1)=0$$
for all finite products with $x_i\in\ker\varepsilon$ and $\sum_in_i\ge N$.

%
%
%


\begin{prop}\label{PropLambda}
Consider $\Lambda$ as in \ref{SetUp}.
For every $\ell\in\mZ_{>0}$, the $\ell$-th Adams operation, restricted to a group homomorphism $\ker\varepsilon\to\ker\varepsilon$, satisfies
$$\varinjlim_{\psi^{\ell}}(\ker\varepsilon)\;\cong\; (\ker\varepsilon)[1/\ell].$$
\end{prop}
\begin{proof}
Consider the $\gamma$-filtration
$$\Lambda=F^0_\gamma\Lambda\;\supset\; F^1_\gamma\Lambda=\ker\varepsilon\;\supset\; F^2_\gamma\Lambda\;\supset\;\cdots\;\supset F^N_\gamma\Lambda=0,$$
with $F^n_\gamma$ the abelian group (the ideal) generated by $\prod_i\lambda^{n_i}(x_i+n_i-1)=0$
for all finite products with $x_i\in\ker\varepsilon$ and $\sum_in_i\ge n$.

It follows from \cite[Propositions 5.4 and 5.5]{AT} or \cite[Proposition~II.4.9]{Weibel} that on $F^n_\gamma/F^{n+1}_\gamma$, with $n\in\mN$, the Adams operation $\psi^\ell$ acts as $\ell^n$. Consequently, the action of $\psi^\ell$ on $\ker\varepsilon$ satisfies the conditions in Lemma~\ref{LemEll}(1), from which the conclusion follows.
\end{proof}

\begin{remark}
\begin{enumerate}
\item Clearly, Proposition~\ref{PropLambda} remains valid for direct limits of $\Lambda$-rings as in \ref{SetUp}, which is a weakening of assumptions.
\item However, the condition that the $\gamma$-filtration be finite cannot be lifted completely. To give a concrete example, consider the $\lambda$-ring
$$R(C_2)= K_0(\Rep_{\mC}C_2)\;\cong\;\mZ[x]/(x^2-1),\quad\mbox{with}\quad \;\ker\varepsilon\;=\; \mZ(1-x)\;\cong\;\mZ,$$
for $\varepsilon: R(C_2)\to\mZ$ coming from the dimension.
Here $\psi^2=0$ on $\ker\varepsilon$, so the conclusion of Proposition~\ref{PropLambda} is not valid. 
\end{enumerate}

\end{remark}

\section{Perfection and $K$-theory}

\subsection{Main result}\label{SecMain}

\subsubsection{} For a commutative ring $R$, we let Let $\widetilde{K}_0(R)$ denote the (split) kernel of the rank \begin{equation}\label{eqH0}K_0(R)\to H_0(R)=[\Spec R,\mZ],\end{equation} see \cite[Definition II.2.3]{Weibel}. For $n>0$, we just set $\widetilde{K}_n=K_n$.

Fix a prime $p$. For a commutative $\mF_p$-algebra $R$, its perfection is defined as
$$R_{\perf}\;:=\; \varinjlim_{Fr}R,$$
where $Fr:R\to R$ is the Frobenius homomorphism $r\mapsto r^p$.

\begin{theorem}\label{Main}
Let $R$ be a commutative $\mF_p$-algebra. For every $n\in\mN$, we have isomorphisms, natural in $R$,
$$\widetilde{K}_n(R_{\perf})\;\cong\; \widetilde{K}_n(R)[1/p].$$
In other words, $\widetilde{K}_n(Fr)$ satisfies the properties in Lemma~\ref{LemEll}.
\end{theorem}

\begin{remark}\label{RemMain}
An alternative formulation of the theorem is
$$K_n(R_{\perf})\cong\begin{cases}
H_0(R)\oplus \widetilde{K}_0(R)[1/p]&\mbox{if $n=0$,}\\
K_n(R)[1/p]&\mbox{if $n>0$.}
\end{cases}
$$
\end{remark}

The remainder of the section is devoted to the proof, which merely consists of gathering classical results and applying the observations from the previous section.

\begin{lemma}\label{Lem1}
Let $R$ be a noetherian commutative ring with finite Krull dimension (e.g. $R$ is finitely generated) and take $n\in \mZ_{>0}$.
The $\lambda$-ring $K_0(R)$, resp. $K_0(R)\times K_n(R)$, has a structure as in \ref{SetUp}, and $\ker\varepsilon =\widetilde{K}_0(R)$, resp. 
$$\ker \varepsilon \;=\; \widetilde{K}_0(R) \times K_n(R).$$
\end{lemma}
\begin{proof}
The $\lambda$-ring structures are defined in \cite{Weibel}, with $\varepsilon$ as in \eqref{eqH0} (and $K_n(R)\subset \ker\varepsilon$). That the $\gamma$-filtration is finite follows from \cite[Th\'eor\`eme~1.ii]{Soule} (and that fact that multiplication is zero on $K_n(R)$) and \cite[Example~II.4.8.2]{Weibel}
\end{proof}

The following is \cite[Proposition~5.4]{Kratzer} or \cite[Ex. II.4.2 and Corollary IV.5.5.2]{Weibel}.

\begin{lemma}\label{Lem2}
Let $R$ be a commutative $\mF_p$-algebra, then the induced action of $Fr$ on $K$-theory corresponds to the Adams operation $\psi^p$ for the $\lambda$-ring structures in Lemma~\ref{Lem1}
\end{lemma}

\begin{proof}[Proof of Theorem~\ref{Main}]
As $K$-theory commutes with direct limits, we have two consequences.
Firstly, by Lemma~\ref{Lem2},
$$\widetilde{K}_\ast(R_\perf)\;\cong\;\varinjlim_{\psi^p} \widetilde{K}_\ast(R).$$
Secondly, it is sufficient to prove the theorem for finitely generated $\mF_p$-algebras. The claim thus follows from application of Proposition~\ref{PropLambda}, thanks to Lemma~\ref{Lem1}.
\end{proof}

\begin{remark}\label{RemDet}
For the sake of completeness, we point out that $K$-theory cannot `detect' whether an $\mF_p$-algebra is perfect. For example, we have
$$K_\ast(\mF_p[t])\;\cong\; K_\ast(\mF_p).$$
On the other hand, when $Fr$ is not injective on $R$, then clearly $K_1(R)$ has $p$-torsion.
\end{remark}

\subsection{Topological localisation}

\subsubsection{}\label{simple}
Following \cite[Chapter 2]{Sullivan}, a topological space $X$ is `simple' if 
\begin{enumerate}
\item $X$ is connected and homotopy equivalent to a CW-complex;
\item $\pi_1(X)$ is abelian;
\item $\pi_1(X)$ acts trivially on $\pi_\ast(\widetilde{X})$ and $H_\ast(\widetilde{X})$ for the universal cover $\widetilde{X}\to X$.
\end{enumerate}

For a prime $p$ and a simple space $X$, the $p'$-localisation (localisation `away from $p$') is a map $X\to X_{\frac{1}{p}}$ of simple spaces which is universal (up to homotopy) with respect to $p'$-local simple spaces. The latter are characterised by $\pi_i$, or equivalently $H_i$, taking values in $\Ab_{p}$ for $i>0$. Existence is proved in \cite[Chapter 2]{Sullivan}.

\begin{theorem}\label{Thm+}
Let $R$ be a commutative $\mF_p$-algebra. We have a homotopy equivalence
$$(B\mathrm{GL}(R_\perf))^+\;\simeq\; ((B\mathrm{GL}(R))^+)_{\frac{1}{p}}.$$
\end{theorem}
\begin{proof}
By Theorem~\ref{Main}, $R\to R_{\perf}$ induces isomorphisms, for all $n>0$,
$$K_n(R)[1/p] =\mZ[1/p]\otimes\pi_n(B\mathrm{GL}(R)^+)\;\xrightarrow{\cong}\; \pi_n(B\mathrm{GL}(R_{\perf})^+)=K_n(R_{\perf}).$$
That $B\mathrm{GL}(R)^+\to B\mathrm{GL}(R_\perf)^+$ is a $p'$-localisation thus follows from~\cite[Theorem~2.1]{Sullivan}.
For completeness we need that $B\mathrm{GL}(R)^+$ is a simple space, see Corollary~\ref{CorSimp} below.
\end{proof}

\begin{lemma}\label{LemHom1}
For a ring $A$ consider $\mathrm{E}(A)=[\mathrm{GL}(A),\mathrm{GL}(A)]<\mathrm{GL}(A)$ the (normal) subgroup generated by elementary matrices. The canonical action of $K_1(R)=\mathrm{GL}(A)/\mathrm{E}(A)$ on homology groups $H_\ast(\E(A);\mZ)$ is trivial.
\end{lemma}
\begin{proof}
Take an arbitrary 
$$x\in H_i(\E(A);\mZ)\simeq \varinjlim_n H_i(\E_n(A);\mZ).$$
We can thus assume $x$ comes from $H_i(\E_n(A);\mZ)$ for some $n$. Now take a coset $g\E(A)$. Potentially after replacing $n$ with a higher number, the representative $g$ belongs to $\GL_n(A)< \GL(A)$. By \cite[Example~III.1.2.1]{Weibel}, we can replace $g$ by an element of $\GL_{2n}(A)<\GL(A)$
which commutes with $\GL_n(A)<\GL(A)$, so in particular with $\E_n(A)$. It thus follows that the action of $g\E(A)$ on $x$ is trivial. 
\end{proof}

\begin{lemma}\label{LemHom2}Let $A$ be a ring.
The action of $K_1(A)=\pi_1(B\GL(A)^+)$ on the singular homology groups of the universal covering space of $B\GL(A)^+$ is trivial.
\end{lemma}
\begin{proof}
By \cite[Definition~VI.1.1(2)]{Weibel}, we can consider instead the action of $K_1(A)$ on the homology with local coefficients 
$$H_\ast(B\GL(A),\mZ K_1(A))\;\simeq\;  H_\ast(E\GL(A)/\E(A),\mZ)\;\simeq\; H_\ast(\E(A);\mZ)$$
which can in turn be interpreted as the group homology from Lemma~\ref{LemHom1}.
\end{proof}

\begin{corollary}\label{CorSimp}
For any ring $A$, the space $B\GL(A)^+$ is simple.
\end{corollary}
\begin{proof}
Conditions~\ref{simple}(1) and (2) are satisfied by construction. For (3), the action on homotopy groups is dealt with in Lemma~\ref{LemHom2}. For homology groups we can use the fact that $B\GL(A)$ is an $H$-space, see \cite[Exercise IV.1.11]{Weibel}, which implies the desired condition (more generally
 it is then known that all Whitehead products are trivial). 
\end{proof}

\subsection{Negative $K$-theory}

It is easy to construct examples of perfect $\mF_p$-algebras for which the strictly negative $K$-groups are not $p$-divisible.
For instance, using Mayer-Vietoris sequences one can prove the perfected version of \cite[Example~III.4.3.1]{Weibel}:
$$K_{-n}((\Delta_n(\mF_p))_\perf)\;\cong\; K_{-n}(\Delta_n(\mF_p))\;\cong\;\mZ,$$
for every $n\in\mZ_{>0}.$
However, we can recover the result on negative $K$-theory of perfections from \cite{EK, KM} as follows.

\begin{prop}\label{PropNeg}
Let $R$ be a commutative $\mF_p$-algebra. For every $i>0$, the morphism $p^i K_{-i}(Fr)$ on $K_{-i}(R)$ satisfies the properties in Lemma~\ref{LemEll}. In particular 
$$K_{-i}(R)[1/p]\;\to\; K_{-i}(R_{\perf})[1/p]$$
is an isomorphism.
\end{prop}

We will use the following well known lemma.

\begin{lemma}\label{LemHae}
For a ring $A$ and $n\le 1$, realise $K_{n-1}(A)$ as a direct summand in $K_n(A[t,t^{-1}])$ as in \cite[III.4]{Weibel}. For $j>0$ let $f_j$ be the algebra endomorphism of $A[t,t^{-1}]$, which preserves $A$ and satisfies $f_j(t)=t^j$. Then $K_{n}(f_j)$ restricts to the summand $K_{n-1}(A)\subset K_n(A[t,t^{-1}])$ as multiplication with $j$.
\end{lemma}
\begin{proof}
First we prove the case $n=1$. In this case, we have
$$K_0(A)\hookrightarrow K_1(A[t,t^{-1}]),\quad [P]\mapsto (A[t,t^{-1}]\otimes_R P,t),$$
from which the claim follows easily. Subsequently, we can use the standard trick of realising 
$$K_{n-1}(A)\subset K_n(A[t,t^{-1}])\cap K_n(A[s,s^{-1}])\subset K_{n+1}(A[t,t^{-1},s,s^{-1}]),$$
to perform iteration.
\end{proof}

\begin{proof}[Proof of Proposition~\ref{PropNeg}]
We realise $Fr$ on $R[t,t^{-1}]$ as the composition of
$r t^i\mapsto r^p t^i$ and $r t^i\mapsto r t^{pi}$.
By naturality and Lemma~\ref{LemHae}, we have a commutative diagram
$$\xymatrix{
K_{-i}(R)\ar[d]^{pK_{-i}(Fr)} \ar@{^{(}->}[rr]&& K_{1-i}(R[t,t^{-1}])\ar[d]^{K_{1-i}(Fr)}\\
K_{-i}(R) \ar@{^{(}->}[rr]&& K_{1-i}(R[t,t^{-1}])
}.$$
By iteration we can then realise $p^i K_{-i}(Fr)$ as a restriction of $K_0(Fr)$ on some commutative $\mF_p$-algebra. It is well-known (but also follows from the above) that $K_{-i}(R)$ actually lives in $\widetilde{K}_0$, so the conclusion follows from Theorem~\ref{Main}.
\end{proof}
\subsection*{Acknowledgements}
The author thanks Elden Elmanto for pointing out \cite{EK} and Christian Haesemeyer for explaining Lemma~\ref{LemHae}.

\end{document}